 \documentclass[12pt]{amsart}
\usepackage{epsfig,color}
\usepackage{amssymb}
\usepackage{caption}

\makeatother

\theoremstyle{definition}

\def\fnum{equation} 
\newtheorem{Thm}[\fnum]{Theorem}

\newtheorem{Que}[\fnum]{Question}
\newtheorem{Lem}[\fnum]{Lemma}
\newtheorem{Con}[\fnum]{Conjecture}

\newtheorem{Pro}[\fnum]{Proposition}

\numberwithin{equation}{section}

\newcommand{\nn}{{\bf{n}}}

\newcommand{\Hess}{{\text {Hess}}}

\def\RR{{\bold R}}

\def\SS{{\bold S}}

\newcommand{\e}{{\text {e}}}

\newcommand{\cC}{{\mathcal{C}}}

\newcommand{\cL}{{\mathcal{L}}}

\newcommand{\cS}{{\mathcal{S}}}

\newcommand{\eqr}[1]{(\ref{#1})}

\title[Analytical properties for degenerate equations]{Analytical properties for degenerate equations}
\author[]{Tobias Holck Colding}%
\address{MIT, Dept. of Math.\\
77 Massachusetts Avenue, Cambridge, MA 02139-4307.}
\author[]{William P. Minicozzi II}%
 
\thanks{The  authors
were partially supported by NSF Grants DMS 1404540 and DMS 1707270.}

\email{colding@math.mit.edu  and minicozz@math.mit.edu}

\begin{document}

\maketitle

\section{Introduction}

By a classical result, solutions of analytic elliptic PDEs, like the Laplace equation, are analytic.     In many instances, the properties that come from being analytic are more important than analyticity itself.
Many important equations are degenerate elliptic and solutions have much 
lower regularity.   Still,
 one may hope that solutions share  properties of analytic functions.  These properties are
  closely connected to important open problems.

    In this survey, we will explain why solutions of an important degenerate elliptic equation  have analytic properties even though the solutions are not even $C^3$.
This equation, known as the {\emph{arrival time equation}}, is
 \begin{align}	\label{e:arrivalu}
	-1 = |\nabla u|\,\text{div}\left( \frac{\nabla u}{|\nabla u|}\right)\,  .
\end{align}
Here $u$ is defined on a compact connected subset of $\RR^{n+1}$ with smooth mean convex boundary and $u$ is constant on the boundary.  Equation \eqr{e:arrivalu} is the prototype for a family of equations, see, e.g., \cite{OsSe},
 used for tracking moving interfaces in complex situations.  These equations have been instrumental in  applications, including semiconductor processing, fluid mechanics, medical imaging, computer graphics, and material sciences.

Even though   solutions  of \eqr{e:arrivalu} are   a priori only in the viscosity sense,  they are
 always twice differentiable by \cite{CM5}, though not necessarily $C^2$; see \cite{CM6}, \cite{H2}, \cite{I}, \cite{KS}. 
Even  when a solution is $C^2$,   it still might not be $C^3$, Sesum, \cite{S}, let alone analytic.  However, solutions have the following property conjectured for analytic functions:
   
  \begin{Thm}	\label{t:thom}
  \cite{CM9}   The Arnold-Thom conjecture holds for  $C^2$ solutions of \eqr{e:arrivalu}.  Namely, if $x(t)$ is a gradient flow line for $u$, then $x(t)$ has finite length and $\frac{x'(t)}{|x'(t)|}$ has a limit.
  \end{Thm}

The theorem applies, for instance, to solutions of \eqr{e:arrivalu} on  closed convex domains since these are $C^2$ by a 1990 result of Huisken, \cite{H1}  (though  not necessarily $C^3$, \cite{S}).

As we will see in Section \ref{s:s1},
the Arnold-Thom conjecture states that $\frac{x'(t)}{|x'(t)|}$ has a limit whenever $u$ is analytic and the gradient flow line itself has a limit.  Thus, Theorem \ref{t:thom} shows that,
in this way, $C^2$ solutions 
of \eqr{e:arrivalu} behave  like analytic functions are expected to.

  \subsection{The arrival time}

The geometric meaning of  \eqr{e:arrivalu} is that the level sets $u^{-1} ( t)$  are mean convex and evolve by mean curvature flow.  
One says that $u$ is the {\emph{arrival time}} since $u(x)$ is  the time the hypersurfaces $u^{-1}(t)$ arrive at $x$ under the mean curvature flow; see
   Chen-Giga-Goto, \cite{ChGG}, Evans-Spruck, \cite{ES}, Osher-Sethian, \cite{OsSe},  and \cite{CM3}.

Conjecturally,  the Arnold-Thom conjecture holds even for solutions that are not $C^2$, but merely twice differentiable:

\begin{Con}	\label{c:con}
\cite{CM9} Lojasiewicz's inequalities and the Arnold-Thom conjecture hold for all solutions of \eqr{e:arrivalu}.
\end{Con}

If so,
this would explain various conjectured phenomena.
 For example, 
  this would imply that the associated mean curvature flow is  singular at only finitely many times as has been conjectured,  \cite{W3}, \cite{AAG}, \cite{Wa}, \cite{M}.

\vskip1mm
  We believe that the principle that solutions of degenerate equations behave as though they are analytic, even when they are not, should be quite general.  For instance, there should be versions for other flows, including Ricci flow.

\subsection{Ideas in the proof}

To explain the ideas in the proof of Theorem \ref{t:thom}, suppose that the unit speed curve $\gamma(s)$  traces out a gradient flow line for $u$ that limits to a critical point $x_0$.
The simplest way to prove that the unit tangent $\gamma_s$ has a limit would be to prove that 
\begin{align}	\label{e:0p4}
	\int |\gamma_{ss}| \, ds < \infty \, .
\end{align}
   However, this is not necessarily true.  It turns out that $\gamma_{ss}$ is better behaved in some directions than in others, depending on   the geometry of the level sets of $u$.

By \cite{CM5}, the level sets of $u$ near $x_0$ are, in a scale-invariant way, converging to either spheres or cylinders.  This comes from a blow up analysis for the singularities of an associated mean curvature flow.
The spherical case is easy to handle and one can show that $|\gamma_{ss}|$ is integrable in this case.  However,  the cylindrical case is more subtle and $\gamma_{ss}$ behaves quite differently.  In particular,   the estimates in the direction of the axis of the cylinder are not strong enough to give \eqr{e:0p4}.  

There is a good reason that the estimates are not strong enough here: the presence of a
``non-integrable'' kernel for a linearized operator.  Here ``non-integrable'' means that there are infinitesimal variations that do not arise as the derivative of an actual one-parameter family of solutions.  As is well known, this corresponds to a slow rate of convergence to the limiting blow up.  Overcoming this requires
 a careful analysis of this kernel using the rate of growth (the frequency function) for the drift Laplacian.

\section{Gradient flows in finite dimensions}	\label{s:s1}

Given a function $f$, a gradient flow line $x(t)$ is a solution of the ODE 
\begin{align}
	x'(t) = \nabla f \circ x(t) 
\end{align}
with the initial condition $x(0) = \bar{x}$.
The chain rule gives that
\begin{align}	\label{e:chain}
	(f \circ x(t))' = |\nabla f|^2  \circ x(t) \, , 
\end{align}
so  we see that $f\circ x(t)$ is increasing unless $x(t) \equiv \bar{x}$  is a critical point of $f$.  

It is possible that $x(t)$ runs off to infinity (e.g., if $f(x,y) = x$ on $\RR^2$), but we are interested in the case where there is a limit point $x_{\infty}$.  That is, where there exist   $t_i \to \infty$ so that $x(t_i) \to x_{\infty}$.  It follows easily that   $\lim_{t \to \infty} f \circ x(t) = f(x_{\infty})$, $x_{\infty} $ is a critical point,  and $|x'|^2$ is integrable.
This raises the obvious question:
\begin{Que}
Does $x(t)$ converge to $x_{\infty}$?  
\end{Que}
Perhaps surprisingly, there are examples where $x(t)$ does not converge; see, e.g.,  fig.~$3.5$ in \cite{Si} or fig.~$1$ in \cite{CM8}.  However, 
if $f$ is real analytic, Lojasiewicz, \cite{L1}, proved that $x(t)$ has finite length and, thus,  converges.    This is  known as {\emph{Lojasiewicz's theorem}}.
The proof relied on two   {\emph{Lojasiewicz  inequalities}} for analytic functions.

 \subsection{Lojasiewicz inequalities}

    In real algebraic geometry, the Lojasiewicz inequality,  \cite{L3},   bounds the distance from a point to the nearest zero of a given real analytic function. Namely,  if $Z \ne \emptyset$ is  the zero set   of $f$ and $K$ is a compact set, then
    there exist   $\alpha\geq 2$ and a positive constant $C$ such that for $x \in K$
  \begin{align}	\label{e:Lj1}
  \inf_{z\in Z} |x-z|^{\alpha}\leq C\, |f(x) | \, .
  \end{align}
The exponent $\alpha$ can be arbitrarily large, depending on the function $f$.

 Equation \eqr{e:Lj1} was the main ingredient in Lojasiewicz's proof of Laurent Schwarz's division conjecture\footnote{L. Schwartz conjectured that if $f$ is a non-trivial real analytic function and $T$ is a distribution, then there exists a distribution $S$ 
satisfying $f\, S$ = $T$.} in analysis.  
Around the same time, 
H\"ormander, \cite{Ho}, independently proved Schwarz's division conjecture in the special case of polynomials and a  key step in his proof was also 
 \eqr{e:Lj1}  when $f$ is a polynomial.

  Lojasiewicz solved a conjecture of Whitney\footnote{Whitney conjectured that if $f$ is analytic  in an 
open set $U$ of $\RR^n$, then
 the zero set $Z$ is a deformation retract of an open neighborhood of $Z$ in 
$U$.} in \cite{L4} using
a second inequality -- known as the gradient inequality:  Given  a critical point $z$, there is  neighborhood $W$ of $z$ and constants $p>1$ and $C> 0$ such that for all $x\in W$
 \begin{align}	\label{e:lou}
  |f(x)-f(z)| \leq C\, |\nabla_x f |^p \, .
  \end{align}
  An immediate consequence of \eqr{e:lou} is that $f$ takes the same value at every critical point in $W$.    
   It is easy to construct smooth functions where this is not the case.
   
   This gradient inequality \eqr{e:lou} was the key ingredient in the proof of the Lojasiewicz theorem.  The idea is that \eqr{e:lou} and \eqr{e:chain} give a differential inequality for $f$ along the gradient flow line that leads to a rate of convergence; see, e.g., \cite{L1}, \cite{CM2}, \cite{CM8} and \cite{Si}.

\subsection{Arnold-Thom conjectures}
 
      Around 1972,   Thom, \cite{T}, \cite{L2}, \cite{Ku}, \cite{A}, \cite{G}, conjectured a strengthening of Lojasiewicz's theorem, asserting that  each gradient flow line $x(t)$ of an analytic function $f$ approaches its limit from a unique limiting direction:
 
 \begin{Con}	\label{c:conj0}
 If   $x(t)$   has a limit point, then 
  the limit of secants $\lim_{t \to \infty} \, \frac{x(t) - x_{\infty} } { |  x(t) - x_{\infty}|}$ exists. 
 \end{Con}

    \begin{figure*}[htbp]
\centering\includegraphics[totalheight=.18\textheight, width=.50\textwidth]{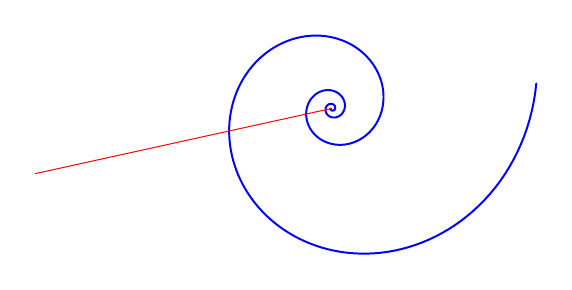}
\caption*{{\bf{Figure}} illustrates in $\RR^3$ a situation   conjectured to be impossible.      The Arnold-Thom conjecture asserts that a blue 
integral curve  does  not spiral  as it approaches the   critical set (illustrated in red,  orthogonal to the plane where the curve spirals). }   
  \end{figure*}

    This 
 conjecture  arose 
 in  Thom's work on catastrophe theory and singularity theory  and became known as {\emph{Thom's gradient conjecture}}.  The conjecture  was finally proven in 2000 by Kurdyka, Mostowski, and Parusinski in \cite{KMP},
  but the following stronger conjecture remains open  (see page $282$ in Arnold's problem list,   \cite{A}):

\begin{Con}  	\label{c:conj}
  If  $x(t)$  has a limit point, then the limit
  of the unit tangents $\frac{x'(t)}{|x'(t)|}$ exists. 
 \end{Con}
 
 \vskip2mm
 \noindent
 It is easy to see that  the {\emph{Arnold-Thom conjecture}}
 \ref{c:conj} implies Thom's  conjecture \ref{c:conj0}.

\section{Lojasiewicz theorem for the arrival time}	\label{s:s2}

The arrival time $u$ is a solution of the degenerate elliptic equation  \eqr{e:arrivalu} and, in particular, it is not  smooth in general, let alone real analytic. 
However, if $u$ is $C^2$, then it satisfies the following gradient Lojasiewicz inequality:

\begin{Thm}	\label{l:gL}
\cite{CM9} 
If $u$ is a $C^2$ solution of  \eqr{e:arrivalu} and $\sup u =0$, then $0$ is the only critical value
 and 
\begin{align}	\label{e:uniqun}
 	\frac{|\nabla u|^2}{-u} \to \frac{2}{n-k} {\text{ as }} u \to 0  \, .
 \end{align}
 In particular, there exists $C>0$ so that $C^{-1} \, |\nabla u|^2 \leq -u \leq C \, |\nabla u|^2$.
 \end{Thm}

 In particular,  \eqr{e:lou} holds with $p=2$ for a   $C^2$ solution $u$ of  \eqr{e:arrivalu}.  Given any  $p>1$, there are solutions of \eqr{e:arrivalu} where \eqr{e:lou}
fails for $p$; obviously, these  are not $C^2$.  Namely, for any odd   $m\geq 3$,     Angenent and Vel\'azquez, \cite{AV}, construct rotationally symmetric examples  with  
\begin{align}
	|u-u(y)| \approx |\nabla u|^{ \frac{m}{m-1}}
\end{align}
 for a sequence of points tending to $y$.  The examples in \cite{AV} were constructed to analyze so-called type II singularities that were previously observed by Hamilton and proven rigorously to exist by Altschuler-Angenent-Giga,   \cite{AAG}; cf. also \cite{GK}.
  
\subsection{The flows lines approach the critical set orthogonally}
 Let $u$ be a $C^2$ solution to \eqr{e:arrivalu} with $\sup u = 0$ and $\cS$   its critical set
\begin{align}
	\cS = \{ x \, | \, \nabla u(x) = 0 \} \, .
\end{align}
 The mean curvature flow given by the level sets of $u$  is smooth away from   $\cS$ and each point in $\cS$ has a cylindrical singularity; see, \cite{W1}, \cite{W2},   \cite{H1}, \cite{HS1}, \cite{HS2}, \cite{HaK}, \cite{An}; cf. \cite{B}, \cite{CM1}.   
Moreover, \cite{CM5} and \cite{CM6} give:
\begin{enumerate}
\item[($\cS 1$)]   $\cS$ is a closed embedded connected $k$-dimensional $C^1$ submanifold whose tangent space is   the kernel of $\Hess_u$.   Moreover, $\cS$ lies in the interior of the region where $u$ is defined.
\item[($\cS 2$)] 
 If $q \in \cS$, then
	$\Hess_u (q) = - \frac{1}{n-k} \, \Pi$ and $\Delta u (q) = - \frac{n+1-k}{n-k}$, 
 where $\Pi$ is orthogonal projection onto the orthogonal complement of the kernel.
 \end{enumerate}

\vskip2mm
The next theorem shows that the gradient flow lines of $u$  have finite length (this is the Lojasiewicz theorem for $u$), converge to points in $\cS$, and approach $\cS$ orthogonally.   The first claims follow immediately from the gradient Lojasiewicz inequality of Theorem \ref{l:gL}.
Let $\Pi_{\text{axis}}$ denote orthogonal projection onto the kernel of $\Hess_u$.

\begin{Thm}	\label{l:flowlines}
\cite{CM9}  Each flow line $\gamma$  for $\nabla u$ has finite length and limits to a point in $\cS$.
 Moreover, if we parametrize $\gamma$ by   $s\geq 0$ with $|\gamma_s| =1$ and $\gamma (0) \in \cS$, then
\begin{align}
	u ( \gamma (s)) & \approx   \frac{-s^2}{2(n-k)}   \, ,  \label{e:flow1}  \\
	\left|\nabla u   ( \gamma (s))  \right|^2 & \approx \frac{s^2}{(n-k)^2} \, ,   \label{e:flow2}  \\
		\Pi_{\text{axis}} (\gamma_s) &\to 0 \, .   \label{e:flow3}
\end{align}
In particular, for $s$ small, we have  that $\gamma (s) \subset B_{ 2n \, \sqrt{ -u(\gamma(s))}}  (\gamma(0))$.
\end{Thm}

\section{Theorem \ref{t:thom} and an estimate for rescaled MCF}

   The Arnold-Thom conjecture for the arrival time is phrased as an analytic question about solutions to a degenerate elliptic partial differential equation.  Yet, we will see that the key is  understanding the geometry of an associated mean curvature flow.  
   
   \subsection{Rescaled mean curvature flow}

A one-parameter family of hypersurfaces $M_{\tau}$ evolves by {\emph{mean curvature flow}} (or {\emph{MCF}}) if each point $x(\tau)$ evolves by 
$\partial_{\tau} x =   -  H  \, \nn$.  Here $H$ is the mean curvature and $\nn$ a unit normal.  The arrival time gives a mean curvature flow 
\begin{align}
	\Sigma_{\tau} = \{ x \, | \, u (x) = \tau \} \, .
\end{align}
As $\tau$ goes to the extinction time (the supremum of $u$), the level sets contract and eventually disappear.  To capture the structure near the extinction, 
we consider the 
rescaled level sets
\begin{align}
	 	\Sigma_t = \frac{1}{\sqrt{-u}} \, \{ x   \, | \, u(x) = - \e^{-t} \} \, .
\end{align}
This is equivalent to simultaneously running MCF and rescaling space and reparameterizing time.
The one-parameter family   $\Sigma_t$
satisfies the {\emph{rescaled  MCF}}   
\begin{align}
	\partial_t x =   - \left( H - \frac{1}{2} \, \langle x , \nn \rangle \right) \, \nn   \, .
\end{align}
The rescaled MCF   is the negative gradient flow for the Gaussian area  
\begin{align}
	F (\Sigma) \equiv \int_{\Sigma} \e^{ - \frac{|x|^2}{4} }   \, .
\end{align} 
In particular, $F(\Sigma_t)$ is non-increasing.

It will be convenient to set $\phi = H - \frac{1}{2} \, \langle x , \nn \rangle$.
The fixed points for rescaled MCF are shrinkers where $\phi =0$;  the most important shrinkers are cylinders
$\cC = \SS^{n-k}_{\sqrt{2(n-k)}} \times \RR^k$ where $k = 0 , \dots , n-1$.    

\subsection{Rate of convergence of the rescaled MCF}

The $F$ functional is nonincreasing along the rescaled MCF $\Sigma_t$ and it is constant only when $\Sigma_t$ is also constant.  Furthermore, the distance between $\Sigma_j$ and $\Sigma_{j+1}$ is bounded by 
\begin{align}	\label{e:deltaj}
	\delta_j \equiv \sqrt{F (\Sigma_{j-1}) - F(\Sigma_{j+2})} \, .
\end{align}
We refer to \cite{CM9} (cf. \cite{CM2}) for the precise statement, but the idea is simple.  To see this, consider the analogous question for a finite dimensional gradient flow $x(t)$.  In this case, the fundamental theorem of calculus and the Cauchy-Schwarz inequality give
\begin{align}
	|x(j+1) - x(j)| \leq \int_j^{j+1} |x'(t)| \, dt \leq \left(  \int_j^{j+1} |x'(t)|^2 \, dt \right)^{ \frac{1}{2} } = \left( f(j+1) - f(j) \right)^{ \frac{1}{2} } \, .
\end{align}

  Existence of $\lim_{t \to \infty} \Sigma_t$ is proven in \cite{CM2} by showing that $\sum \delta_j < \infty$.  In \cite{CM9}, we prove that
  $\delta_j$ is summable even after being raised to some power  less than one:

\begin{Pro}	\label{c:djsums1}
\cite{CM9} There exists $\bar{\beta} < 1$ so that  $\sum_{j=1}^{\infty} \, \delta_j^{ \bar{\beta}} < \infty $.
\end{Pro}

\subsection{A strong cylindrical approximation}

Since  $\Sigma_t$ converges to a limit $\cC$, $\Sigma_j$ is  close to   $\cC$ for $j$ large.  However,  we will construct cylinders $\cC_j$, varying with $j$, that are even closer.
    We need some notation:
 $\Pi_j$ is projection orthogonal to axis of $\cC_j$,  $\cL$ is the drift Laplacian on  $\cC_j$, and
 \begin{align}
  	\| g \|_{L^p (\Sigma_j)}^p \equiv  \int_{\Sigma_j} |g|^p \, \e^{ - \frac{|x|^2}{4} }  \, .
\end{align}
The precise statement of the approximation  is technical (see \cite{CM9}), but  it roughly says:

\begin{Pro}	 \label{p:evolving}
\cite{CM9} Given   $ \beta < 1$, there exist $C $,   radii $R_j$, and cylinders $\cC_j$ with:
\begin{enumerate}
\item For $t \in [j,j+1]$, $\Sigma_t$ is a graph over $B_{R_j} \cap \cC_{j+1}$ of  a function $w$ with 
\begin{align}
	 \| w \|_{ W^{3,2}}^2 + \| \phi \|_{W^{3,2}(B_{R_j})}+ \e^{ - \frac{R_j^2}{4} } &\leq C  \, \delta_j^{\beta} \, . \notag
\end{align}
\item  
 $w$ is almost an eigenfunction; i.e.,  $\left| \phi - \left( \cL + 1 \right) w \right|$ is quadratic in $w$.
 
 \item $\left| \Pi_j - \Pi_{j+1} \right| \leq C  \,\delta_j^{\beta}$. 
\item The higher derivatives of $w$ and $\phi$ are bounded.
\end{enumerate}

\end{Pro}

\subsection{Reduction}
 The next theorem reduces  Theorem \ref{t:thom}  to an estimate for  rescaled MCF.

 \begin{Thm}	\label{t:rMCFA}
\cite{CM9} Theorem \ref{t:thom}  holds if 
\begin{align}	\label{e:rmcfa}
	\sum_{j=1}^{\infty} \,   \int_j^{j+1} \left( \sup_{B_{2n} \cap \Sigma_t} \left| \Pi_{j+1} (\nabla H) \right| \right) \, dt  < \infty \, .
\end{align}
 \end{Thm}
 
 \vskip1mm
 To explain Theorem \ref{t:rMCFA}, let $\gamma (s)$ be a unit speed parameterization of a gradient flow line   with   $\gamma (0)   \in \cS$.   We will show that $\gamma_s$ has a limit as $s \to 0$.
 The derivative of $\gamma_s = - \frac{\nabla u}{|\nabla u|}$ is  
\begin{align}	\label{e:D3a}
 	\gamma_{ss} &=  - \frac{1}{|\nabla u|} \, \left(  \Hess_u ( \gamma_s) -  \gamma_s \, \langle  \Hess_u ( \gamma_s) , \gamma_s  \rangle	\right)= - \frac{\left( \Hess_u (\gamma_s) \right)^T }{|\nabla u|}=
	\nabla^T \log |\nabla u|   \, ,
 \end{align}
 where $\left( \cdot \right)^T$ is the tangential projection onto the level set of $u$.  
 
 \vskip2mm
 The simplest way to prove that $\lim \gamma_s$ exists would be to show that $\int |\gamma_{ss}| < \infty$, which is related to the rate of convergence for an associated
  rescaled MCF.  While this rate   fails to give integrability of $|\gamma_{ss}|$, it does give the following:
  
  \begin{Lem}	\label{l:seq}
  Given any $\Lambda > 1$, we have
$
  	\lim_{s \to 0} \,  \int_{s}^{\Lambda \, s} |\gamma_{ss} | \, ds  = 0$.
	  \end{Lem}
  
  \begin{proof}
 Using Theorem \ref{l:flowlines} and the fact that $\Hess_u \to - \frac{1}{n-k} \, \Pi$, 
  \eqr{e:D3a} implies that $  s\, |\gamma_{ss} |   \to 0$.  The lemma follows immediately from this.
  \end{proof}
   
  To get around the lack of integrability, we will decompose $\gamma_s$ into two pieces - the parts tangent and orthogonal to the axis - and deal with these separately.  The tangent part goes to zero by
  \eqr{e:flow3} in Theorem \ref{l:flowlines}.  We will use  \eqr{e:rmcfa} to  control  the orthogonal part.

Translate so that $\gamma (0) = 0$ and
let $\bar{H}= \frac{1}{|\nabla u|}$ be the mean curvature of the level set of $u$.
The mean curvature $H$  of   $\Sigma_t$ at time $t = - \log (-u)$ is given by
 \begin{align}		\label{e:barH}
 	\bar{\nabla}  \log \bar{H}  = \frac{  {\nabla} \log  {H}  }{ \sqrt{ - u  }}   \approx \frac{ \sqrt{2(n-k)}}{s} \, \ {\nabla} \log {H}   \, .
 \end{align}
    Note that  $u(\gamma (s))$ is decreasing and Theorem \ref{l:flowlines} gives 
\begin{align}	  	
	  t(s) & \approx - 2 \, \log s + \log (2(n-k))  \, , \\
 	t'(s) &=- \partial_s \, \left( \log (-u(\gamma(s)) \right) = \frac{ -\partial_s u(\gamma(s))}{ u(\gamma(s)) } \approx - \frac{2}{s} \, . \label{e:tofs2}
 \end{align}
 Given a positive integer $j$, define $s_j$ so that $t(s_j) = j$.  Note that  $\left| \log \frac{s_{j+1}}{s_j} \right|$ is uniformly bounded.  
 Therefore, by Lemma \ref{l:seq}, it suffices to show that $\gamma_{s_j}$ has a limit.  
 
 We can write $\gamma_{s_j} =\Pi_{\text{axis},j} (\gamma_{s_j}) + \Pi_j (\gamma_{s_j})$.  We have $\Pi_{\text{axis},j} (\gamma_{s_j})  \to 0$ since
 $\Pi_{\text{axis},j} \to \Pi_{\text{axis}}$ and $\Pi_{\text{axis}} (\gamma_s) \to 0$.  Thus,  we need that
 $\lim_{j \to \infty} \Pi_j (\gamma_{s_j}) $ exists; this will follow   from
 \begin{align}	\label{e:toprovej}
 	\sum_j \, \left| \Pi_j (\gamma_{s_j}) -  \Pi_{j+1} (\gamma_{s_{j+1}}) \right| < \infty \, .
 \end{align}

 Theorem \ref{l:flowlines} gives (for $s$ small) that $\gamma (s) \subset B_{ 2n \, \sqrt{ -u(\gamma(s))}}  $ and, thus,
  \eqr{e:D3a} gives
\begin{align}	 
	\left| \Pi_{j+1} (\gamma_{s_j}) -  \Pi_{j+1} (\gamma_{s_{j+1}}) \right|  &\leq \int_{s_{j+1}}^{s_j} \left| \Pi_{j+1} (\gamma_{ss} ) \right| \, ds = \int_{s_{j+1}}^{s_j}  \left| \Pi_{j+1}\left( \bar{\nabla}  \log \bar{H} (\gamma(s)) \right) \right| \, ds  \notag \\
	& \leq C \, \int_{s_{j+1}}^{s_j}  \sup_{B_{ 2n \, \sqrt{-u(\gamma(s))}}}\, \left| \Pi_{j+1} (\nabla \log  \bar{H}) \right| ( \cdot ,-u) \, ds  \, .	\label{e:aaaa}
\end{align} 
Using \eqr{e:barH} and \eqr{e:tofs2} in \eqr{e:aaaa} and then applying Theorem \ref{t:rMCFA} gives
\begin{align}	 
	\sum_j \, \left| \Pi_{j+1} (\gamma_{s_j}) -  \Pi_{j+1} (\gamma_{s_{j+1}}) \right| &   
	 \leq C \, \sum_j \,  \int_j^{j+1} \sup_{B_{ 2n} \cap \Sigma_t } \, \left| \Pi_{j+1} (\nabla   {H}) \right|  \, dt < \infty \, .
\end{align} 
  On the other hand, $\sum_j \left| \Pi_{j} (\gamma_{s_j}) -  \Pi_{j+1} (\gamma_{s_{j}}) \right|   < \infty$ by (3) in Proposition \ref{p:evolving} and Proposition \ref{c:djsums1}.
 	 The triangle inequality gives \eqr{e:toprovej}, so we conclude that $\gamma_s$ has a limit.

\subsection{The summability condition \eqr{e:rmcfa}}

We have seen that the key    is to prove \eqr{e:rmcfa}.   This summability is plausible  since $\Sigma_t$ is converging to a cylinder where $H$ is constant and, thus, $\nabla H$ is going to zero.  The rate of convergence then becomes critical.  If the convergence was fast enough, then $|\nabla H|$ would be summable even without the projection $\Pi_{j+1}$.

The mean curvature $H$ of the graph of $w$ is given at each point explicitly as a function of $w$, $\nabla w$ and $\Hess_w$; see
          corollary $A.30$ in \cite{CM2}.  We can write this as the first order part (in $w, \nabla w , \Hess_w$) plus a quadratic remainder
          \begin{align}	\label{e:expandH}
          	H   &= H_{\cC} + \left( \Delta_{\theta} + \Delta_x +   \frac{1}{2} \right) w + O(w^2)     \, .
          \end{align}
           Here  $O(w^2)$ is a term that depends at least quadratically on $w, \nabla w , \Hess_w$ and
            the constant $H_{\cC} = \frac{ \sqrt{n-k} }{\sqrt{2}}$ is the mean curvature of $\cC$.

  The bound for $w^2$ in (1) from Proposition \ref{p:evolving} is  summable by Proposition \ref{c:djsums1}, but the bound for $w$   is not.     In particular, (1) gives a bound for $\nabla H$ that is not summable.

\section{Approximate eigenfunctions on cylinders}

Proposition \ref{p:evolving} shows the graph function $w$ is an approximate eigenfunction on the cylinder $\cC$.  Namely, (2) gives that 
\begin{align}
	\left| (\cL +1 ) w - \phi \right| = O(w^2) \, , 
\end{align}
where $O(w^2)$ is a term that is quadratically bounded in $w$ and its derivatives.  Note that $\phi$ itself is bounded by (1) and, moreover, $\phi$ is of the same order as $w^2$.

Even though the bound for $w$ is not summable, we will see that there is a function $\tilde{w}$ in the kernel of $\cL + 1$ so that $|w- \tilde{w}|$ is summable.  Moreover,  the contribution of  $\tilde{w}$   to $\nabla H$  goes away once we project orthogonally to the axis.  Putting this together gives \eqr{e:rmcfa} and, thus, completes the proof of Arnold-Thom.  The arguments needed for this decomposition of $w$ are technically complicated because of higher order ``error'' terms; see \cite{CM9}.  However, the idea   is clear.  We will explain this in a model case next.

\subsection{Eigenfunctions on cylinders}

The eigenfunctions on the cylinder are built out of spherical eigenfunctions on the cross-section and eigenfunctions for the Euclidean drift Laplacian on the axis.  
Namely, by lemma $3.26$ in \cite{CM2}, the kernel of  $\cL +1$ on the weighted Gaussian space on $\cC$ consists of quadratic polynomials and ``infinitesimal rotations''  
\begin{align}	\label{e:kernelcC}
	\tilde{w} = \sum_i a_i (x_i^2 - 2) + \sum_{i<j} a_{ij} x_i x_j + \sum_k   x_k h_k (\theta) \, , 
\end{align}
where $a_i , a_{ij}  $ are constants and each $h_k $ is a $\Delta_{\theta}$-eigenfunction with eigenvalue $\frac{1}{2}$.

\vskip2mm
To illustrate the ideas involved, it is helpful to recall the Euclidean case:

\begin{Lem}	\label{l:hermite}
If $\cL v = - \lambda v$ on $\RR^n$ and $\int_{\RR^n} v^2 \, \e^{ - \frac{|x|^2}{4} } < \infty$, then $2\, \lambda$ is a nonnegative integer and $v$ is a polynomial of degree $2\lambda$.
\end{Lem}

When $n=1$, these  are the Hermite polynomials (up to a scaling normalization).  

\begin{proof}[Sketch of the proof of Lemma \ref{l:hermite}] There are two ingredients:   
 \begin{itemize}
 \item Each partial derivative $v_i = \frac{\partial v}{\partial x_i}$ satisfies $\cL \, v_i = \left( \frac{1}{2} - \lambda \right) \, v_i$.
 \item $\int_{\RR^n} |\nabla v|^2 \, \e^{ - \frac{|x|^2}{4} } \leq 2\lambda \, \int_{\RR^n} v^2 \, \e^{ - \frac{|x|^2}{4} } < \infty$.
 \end{itemize}
The second property implies that $\lambda \geq 0$ and $v$ is constant if $\lambda = 0$.  The lemma follows by applying this to $2\lambda$  derivatives of $v$.
 \end{proof}
  
 \vskip1mm
  The next theorem   approximates $w$ in $|x| \leq 3n$ by   $\tilde{w}$ as in \eqr{e:kernelcC} (we state the theorem in the model case where $w$ is an eigenfunction; see \cite{CM9} for approximate eigenfunctions). 
  
\begin{Thm}	\label{t:goalCM7}
\cite{CM9} Given $\nu < 1$, there exists $C$ so that if $(\cL +1)w =0$ on $B_R$ with $\e^{ - \frac{R^2}{4}} + \| w \|_{W^{3,2}}^2 \leq \delta$ and
$w^2 \leq \delta \, \e^f$, then there is a function $\tilde{w}$ as in \eqr{e:kernelcC} with
\begin{align}	\label{e:tgoal}
	\sup_{|x| \leq 3n} \, \left| w - \tilde{w} \right| \leq C \,\delta^{ \nu }  \, .
\end{align}
\end{Thm}

This gives the improved estimate that we need.  
 Namely,   (1) in Proposition \ref{p:evolving} gives $|w  | \leq C \, \delta_j^{ \frac{1}{2} }$,   while \eqr{e:tgoal} gives $|w - \tilde{w}| \leq C \, \delta_j^{ \nu }$ with $\nu \approx 1$.  The first bound is not summable, but the second bound is by Proposition \ref{c:djsums1}.

 \vskip2mm
 The $L^2$ methods  for Lemma \ref{l:hermite} yield sharp global results, but  are not sharp enough for \eqr{e:tgoal}.   We will need a different approach  -  the frequency   - that is explained  next.

\subsection{The frequency}

The key to understanding the growth of eigenfunctions for $\cL$ is a frequency function inspired by Almgren's frequency for harmonic functions, \cite{Al}, cf. \cite{GL}, \cite{HaS}, \cite{Ln}, \cite{CM10}, \cite{D}.  The frequency was  used by Bernstein, \cite{Be}, to study the ends of shrinkers and in \cite{CM7} to study the growth of approximate eigenfunctions.

 To explain the frequency,  set $f = \frac{|x|^2}{4}$ on $\RR^n$ and
 define quantities $I(r)$ and $D(r)$
\begin{align}
	I (r) &= r^{1-n} \int_{\partial B_r} u^2 \, , 	\label{e:defI}\\
D(r)& =  r^{2-n} \int_{\partial B_r} u u_r  =r^{2-n} \, \e^{  f(r)} \, \int_{B_r}  \left(|\nabla u|^2-V\,u^2\right) \, \e^{-f}  \, ,\label{e:defD}
\end{align}
and the {\it{frequency}} $U(r)= \frac{D}{I}$.  Thus, 
  $(\log I)' = \frac{2U}{r}$, so $U$ measures the polynomial rate of growth of $\sqrt{I}$.   For example,
if $u(x) = |x|^d$, then $U= d$.

  \vskip1mm
    There is a dichotomy where eigenfunctions of $\cL$ are polynomial or grow  exponentially:

\begin{Thm}   \label{t:main}
\cite{CM7}  Given $\epsilon>0$ and $\delta>0$, there exist $r_1>0$ such that if  $\cL u = - \lambda u$  and $U(\bar{r}_1)\geq \delta+2\,\sup \, \{ 0 ,   \lambda \} $ for some $\bar{r}_1 \geq r_1$, then for all $r\geq R(\bar{r}_1)$
\begin{align}     
U(r) &> \frac{r^2}{2}-n-2\,\lambda -\epsilon\, .   
\end{align}
\end{Thm}

This theorem explains why we expect a good  approximation when the eigenfunction is defined (and bounded) on a large ball.  This is easiest to explain in the case $\lambda =0$ (we can reduce to this after taking $2\lambda$ partial derivatives).  Namely, subtracting a constant to make the average zero on a ball, we can use the Poincar\'e inequality to get a positive lower bound for the frequency on a fixed inner ball.  Theorem \ref{t:main} then implies extremely rapid growth out to the boundary of the ball.  Since
the function is bounded, we conclude that it must be   small on the inner ball as claimed.  See theorem $4.1$ in \cite{CM7} for details (compare      \cite{CM9}).

\end{document}